\documentclass[11pt, a4paper]{amsart}

\usepackage{amsfonts,amsmath,amssymb, amscd}

\newtheorem{theorem}{Theorem}[section]
\newtheorem{lemma}[theorem]{Lemma}

\newtheorem{prop}[theorem]{Proposition}
\newtheorem{corollary}[theorem]{Corollary}

\theoremstyle{definition}
\newtheorem{rem}[theorem]{Remark}
\newtheorem{rems}[theorem]{Remarks}
\newtheorem{Qu}[theorem]{Question}

\newcommand\pf{\begin{proof}}
\newcommand\epf{\end{proof}}

\newcommand\ab{\mathrm{ab}}
\newcommand\co{\mathrm{co}}

\newcommand\Forms{\mathrm{Forms}}
\newcommand\eps{\varepsilon}
\newcommand\Frac{\operatorname{Frac}}
\newcommand\Alg{\operatorname{Alg}}
\newcommand\Spec{\operatorname{Spec}}

\renewcommand\AA{\mathcal A}
\newcommand\BB{\mathcal B}
\newcommand\ZZ{\mathbb{Z}}
\newcommand\mm{\mathfrak{m}}

\DeclareMathOperator{\Ker}{Ker}
\DeclareMathOperator{\id}{id}

\numberwithin{equation}{section}

\title[Properties of the generic Hopf Galois extensions]
{Flatness and freeness properties of the generic Hopf Galois extensions}

\author{Christian Kassel}
\address{Christian Kassel: 
Institut de Recherche Math\'e\-ma\-tique Avanc\'ee,
CNRS \& Universit\'e de Strasbourg,
7 rue Ren\'{e} Descartes, 67084 Strasbourg, France}
\email{kassel@math.unistra.fr}
\urladdr{www-irma.u-strasbg.fr/\raise-2pt\hbox{\~{}}kassel/}

\author{Akira Masuoka}
\address{Akira Masuoka: 
Institute of Mathematics, 
University of Tsukuba, 
Ibaraki 305-8571, Japan}
\email{akira@math.tsukuba.ac.jp}

\dedicatory{Dedicado a Hans-J\"urgen Schneider con respeto y admiraci\'on}

\begin{document}

\begin{abstract}
In previous work, 
to each Hopf algebra~$H$ and each invertible two-cocycle~$\alpha$,
Eli Aljadeff and the first-named author attached a subalgebra~$\BB_H^{\alpha}$ 
of the free commutative Hopf algebra~$S(t_H)_{\Theta}$ 
generated by the coalgebra underlying~$H$;
the algebra~$\BB_H^{\alpha}$ is the subalgebra of coinvariants of a generic Hopf Galois extension. 
In this paper we give conditions under which $S(t_H)_{\Theta}$ is faithfully flat,
or even free, as a $\BB_H^{\alpha}$-module. 
We also show that $\BB_H^{\alpha}$ is generated as an algebra
by certain elements arising from the theory of polynomial identities for comodule algebras
developped jointly with Aljadeff.
\end{abstract}

\maketitle



\hspace{3cm}

\section*{Introduction}

The aim of the present article is to answer two questions left unanswered in~\cite{AK}
and in the follow-up paper~\cite{Ka1}.
Eli Aljadeff and the first-named author associated
to any Hopf algebra~$H$ and any invertible two-cocycle~$\alpha$
a subalgebra~$\BB_H^{\alpha}$ of the free commutative Hopf algebra~$S(t_H)_{\Theta}$ 
generated by the coalgebra underlying~$H$;
the commutative algebra~$\BB_H^{\alpha}$ is the subalgebra of coinvariants
of the \emph{generic Hopf Galois extension}~$\AA_H^{\alpha}$
introduced in~\cite{AK}.

In~\cite[Sect.~7]{AK} it was proved that if 
$S(t_H)_{\Theta}$ is integral over the subalgebra~$\BB_H^{\alpha}$,
then for any field extension~$K/k$, there is a surjective map
\begin{equation}\label{surj}
\Alg(\BB_H^{\alpha},K) \to \Forms_K({}^{\alpha} H) \, 
\end{equation}
from the set of algebra morphisms $\BB_H^{\alpha} \to K$ to the set of isomorphism classes
of $K$-forms of the twisted algebra~${}^{\alpha} H$
(for details, see \emph{loc.\ cit.} or Section~\ref{basic} below).

The first question raised in~\cite{Ka1} (Question~6.1) was: 
under which condition on the pair~$(H,\alpha)$ is
$S(t_H)_{\Theta}$ integral over~$\BB_H^{\alpha}$?
Now, the existence and the surjectivity of the map~\eqref{surj} 
hold under the less restrictive condition
that $\BB_H^{\alpha} \hookrightarrow S(t_H)_{\Theta}$ be \emph{faithfully flat}.
In the present paper we shall give sufficient conditions for the latter to hold.

After having defined in Section~\ref{basic} the objects of interest to us,
in particular the algebras $S(t_H)_{\Theta}$ and~$\BB_H^{\alpha}$, 
we reduce the question to the case where~$\alpha = \eps$ is the trivial two-cocycle
and to the corresponding subalgebra~$\BB_H = \BB_H^{\eps}$ (see Section~\ref{reduction}).

In Section~\ref{results} we show that the $\mathcal{B}_H$-module $S(t_H)_{\Theta}$ is
\begin{itemize}

\item[(i)] 
a \emph{finitely generated projective generator} 
if $H$ is finite-dimensional (see Theorem~\ref{thm-finite}), 
 
\item[(ii)] 
\emph{faithfully flat} if $H$ is cocommutative (see Theorem~\ref{thm-cocom}), and
 
\item[(iii)] 
\emph{free} if
\begin{itemize}
\item[(a)] $H$ is a pointed Hopf algebra whose grouplike elements 
satisfy a certain finiteness condition
(see Theorem~\ref{thm-pointed}), or
\item[(b)] $H$ is commutative or pointed cocommutative (see Theorem~\ref{thm-com}).
\end{itemize}
\end{itemize}
In all these cases, $\BB_H \hookrightarrow S(t_H)_{\Theta}$ is faithfully flat.
These results extend by far Theorem~6.2 of~\cite{Ka1}, which in particular implies 
that the above integrality condition is
satisfied by any finite-dimensional Hopf algebra generated by grouplike 
and skew-primitive elements.

The algebra~$\BB_H$ is generated by the values of the so-called generic cocycle and of its inverse
(see Section~\ref{base}). In~\cite{AK} a theory of polynomial identities
for $H$-comodule algebras was worked out. Such identities live in a tensor algebra~$T(X_H)$.
Certain coinvariant elements of~$T(X_H)$ give rise to natural elements 
$p_x$, $p'_x$, $q_{x,y}$, $q'_{x,y}$ of~$\BB_H$.
The second question we address here is the following: do these elements generate~$\BB_H$? 
In Section~\ref{gens} we give a positive answer for any Hopf algebra
(the answer had been known so far only for group algebras and the four-dimensional Sweedler algebra).

Throughout the paper we fix a field~$k$ over which all
our constructions are defined. 
In particular, all linear maps are supposed to be $k$-linear
and unadorned tensor products mean tensor products over~$k$.

For us an algebra is an associative unital $k$-algebra
and a coalgebra is a coassociative counital $k$-coalgebra. 
We denote the coproduct of a coalgebra by~$\Delta$ and its counit by~$\eps$.
We shall also make use of a Heyneman-Sweedler-type notation 
for the image 
$$\Delta(x) = x_1 \otimes x_2$$
of an element~$x$ of a coalgebra~$C$
under the coproduct, and we write
$$\Delta^{(2)}(x) = x_1 \otimes x_2 \otimes x_3$$
for the iterated coproduct 
$\Delta^{(2)} = (\Delta \otimes \id_C) \circ \Delta = (\id_C \otimes \Delta) \circ \Delta$,
and so~on. 

For any Hopf algebra~$H$ we denote its antipode by~$S$
and its group of grouplike elements by~$G(H)$.

\section{The basic objects}\label{basic}

We now define the basic objects investigated in this paper.

\subsection{The free commutative Hopf algebra generated by a coalgebra}\label{Takeuchi}

Let $C$ be a coalgebra. Pick another copy~$t_C$ of the underlying vector space of~$C$
and denote the identity map from $C$ to~$t_C$ by $x\mapsto t_x$ ($x\in C$).

Let $S(t_C)$ be the symmetric algebra over the vector space~$t_C$.
If $\{x_i\}_{i\in I}$ is a linear basis of~$C$, then $S(t_C)$ is
isomorphic to the polynomial algebra over the indeterminates $\{t_{x_i}\}_{i\in I}$.
The commutative algebra~$S(t_C)$ has a bialgebra structure with coproduct 
$\Delta : S(t_C) \to S(t_C) \otimes S(t_C)$ and counit $\eps : S(t_C) \to k$ induced
by the coproduct and the counit of~$C$, respectively, namely
\begin{equation}\label{comm-coalg}
\Delta(t_x) = t_{x_1} \otimes t_{x_2} \quad\text{and}\quad
\eps(t_x) = \eps(x)
\end{equation}
for all $x\in C$.
The bialgebra~$S(t_C)$ is the \emph{free commutative bialgebra} generated by~$C$:
this means that for any coalgebra morphism $f: C \to B$ from~$C$ to a commutative bialgebra~$B$,
there is a unique bialgebra morphism $\bar f : S(t_C) \to B$ such that
$\bar f(t_x) = f(x)$ for all $x\in C$.

By~\cite[Lemma~A.1]{AK}
there is a unique linear map $x \mapsto t^{-1}_x$
from~$C$ to the field of fractions~$\Frac S(t_C)$ of~$S(t_C)$
such that for all $x\in C$,
\begin{equation*}\label{tt}
t_{x_1} \, t^{-1}_{x_{2}} = t^{-1}_{x_{1}} \,  t_{x_{2}} = \eps(x) \, 1 \, .
\end{equation*}
As in~\cite[App.~B]{AK} we denote by~$S(t_C)_{\Theta}$ the subalgebra of~$\Frac S(t_C)$ generated 
by all elements~$t_x$ and~$t^{-1}_x$. 
The commutative algebra~$S(t_C)_{\Theta}$ has a Hopf algebra structure with coproduct 
$\Delta : S(t_C)_{\Theta} \to S(t_C)_{\Theta} \otimes S(t_C)_{\Theta}$, 
counit $\eps : S(t_C)_{\Theta} \to k$, and antipode $S : S(t_C)_{\Theta} \to S(t_C)_{\Theta}$ 
determined by~\eqref{comm-coalg}, by
\begin{equation}\label{comm-Hopf}
\Delta(t^{-1}_x) = t^{-1}_{x_2} \otimes t^{-1}_{x_1} \quad\text{and}\quad
\eps(t^{-1}_x) = \eps(x) \, ,
\end{equation}
and by
\begin{equation}\label{ant-comm-coalg}
S(t_x) = t^{-1}_x \quad\text{and}\quad 
S(t^{-1}_x) = t_x
\end{equation}
for all $x\in C$.
The Hopf algebra~$S(t_C)_{\Theta}$ is the \emph{free commutative Hopf algebra} generated by~$C$,
as constructed by Takeuchi in~\cite{Ta1}:
more precisely, for any coalgebra morphism $f: C \to H$ from~$C$ to a commutative Hopf algebra~$H$,
there is a unique Hopf algebra morphism $\tilde f : S(t_C)_{\Theta} \to H$ such that
$\tilde f(t_x) = f(x)$ and $\tilde f(t^{-1}_x) = S(f(x))$ for all $x\in C$.

As an example, consider the coalgebra~$C=kX$ freely spanned by a set~$X$ 
of grouplike elements.
The resulting commutative Hopf algebra $S(t_C)_{\Theta}$ is 
isomorphic to the group algebra~$k\ZZ^{(X)}$, 
where $\ZZ^{(X)}$ is the free abelian group with basis~$X$.
Note that $k\ZZ^{(X)}$ is isomorphic to the Laurent polynomial algebra
$k[\, t_x^{\pm 1} \, | \, x\in X\, ]$.

\subsection{Cocycles}\label{cocycle}

Let $H$ be a Hopf algebra.
Recall that a \emph{two-cocycle}~$\alpha$ on~$H$ is 
a bilinear form $\alpha : H \times H \to k$ such that
\begin{equation*}
\alpha(x_1,y_1)\, \alpha(x_2 y_2, z)
= \alpha(y_1, z_1)\, \alpha(x, y_2 z_2)
\end{equation*}
for all $x,y,z \in H$.
We assume that $\alpha$ is convolution-invertible and write~$\alpha^{-1}$ for its inverse.

We denote by~${}^{\alpha}H$ the \emph{twisted $H$-comodule algebra} defined as follows:
as a right $H$-comodule, ${}^{\alpha}H = H$; as an algebra it is equipped with the product
\begin{equation*}
(x, y) \mapsto \alpha(x_1, y_1) \, x_2 y_2 \, ,
\end{equation*}
where $x,y \in H$. The algebra~${}^{\alpha}H$ is a right $H$-comodule algebra
whose subalgebra of $H$-coinvariants consists of the multiple scalars of the unit:
\begin{equation*}
({}^{\alpha}H)^{\co-H} = k\, 1 \, .
\end{equation*}
 
Two-cocycles $\alpha, \beta : H \times H \to k$ are said
to be \emph{cohomologous} if 
there is a  convolution-invertible linear form $\lambda : H \to k$ 
with inverse~$\lambda^{-1}$ such that
\begin{equation}\label{cohomologous}
\beta(x,y) =  \lambda(x_1) \, \lambda(y_1) \, \alpha(x_2,y_2) \, \lambda^{-1}(x_3y_3)
\end{equation}
for all $x, y \in H$.
If $\alpha, \beta, \lambda$ are related as in the previous equation, 
then the map
\begin{equation*}
{}^{\alpha}H \longrightarrow {}^{\beta}H \, ; \; x \longmapsto \lambda^{-1}(x_1) \, x_2
\end{equation*}
is an isomorphism of $H$-comodule algebras.

\subsection{The generic cocycle and the generic base algebra}\label{base}

Now, to a pair $(H,\alpha)$ consisting of a Hopf algebra~$H$ and 
a convolution-invertible two-cocycle~$\alpha$,
we attach a bilinear map
$\sigma : H \times H \to S(t_H)_{\Theta}$
with values in the previously defined algebra~$S(t_H)_{\Theta}$.
The map~$\sigma$ is given for all $x,y \in H$ by
\begin{equation*}\label{sigma-def}
\sigma(x,y) = t_{x_1} \, t_{y_1} \, \alpha(x_2,y_2) \, t^{-1}_{x_3 y_3} \, .
\end{equation*} 
The map $\sigma$ is a two-cocycle of~$H$ with values in~$S(t_H)_{\Theta}$;
by construction, $\sigma$~is cohomologous to~$\alpha$ over~$S(t_H)_{\Theta}$.
We call $\sigma$ the \emph{generic cocycle attached to}~$\alpha$.

The cocycle $\alpha$ being invertible, so is~$\sigma$,
with inverse $\sigma^{-1}$ given for all $x,y \in H$ by
\begin{equation*}\label{sigma^{-1}-def}
\sigma^{-1}(x,y) = t_{x_1 y_1} \, 
\alpha^{-1}(x_2,y_2) \, t^{-1}_{x_3}  \, t^{-1}_{y_3}\, .
\end{equation*}

Following~\cite[Sect.~5]{AK} and~\cite[Sect.~3]{Ka1},
we define  the \emph{generic base algebra}
as the subalgebra~$\BB_H^{\alpha}$ of~$S(t_H)_{\Theta}$ 
generated by the values of the generic cocycle~$\sigma$ and 
of its inverse~$\sigma^{-1}$.

Since $\BB_H^{\alpha}$ is a subalgebra of~$S(t_H)_{\Theta}$, 
the latter is a module over the former. 
In Section~\ref{results} we shall discuss when such a module is free, projective, or flat.

\subsection{The generic Galois extension}\label{Galois}

Since the generic cocycle~$\sigma$ takes its values in~$\BB_H^{\alpha}$,
we can consider the twisted $H$-comodule algebra
\[
\AA_H^{\alpha} = \BB_H^{\alpha} \otimes {}^{\sigma} H\, ,
\] 
which is equal to~$\BB_H^{\alpha} \otimes H$  as a right $H$-comodule,
and whose product is given for all $b$, $c\in \BB_H^{\alpha}$ and $x$, $y \in H$ by
\begin{equation*}
(b\otimes x)  (c\otimes y)
= bc \, \sigma(x_1, y_1) \otimes x_2 y_2 \, .
\end{equation*}
The $H$-comodule algebra~$\AA_H^{\alpha}$ is a $H$-Galois extension
with~$\BB_H^{\alpha} = \BB_H^{\alpha} \otimes 1$ as subalgebra of coinvariants.
We call~$\AA_H^{\alpha}$ the \emph{generic Galois extension}
attached to the cocycle~$\alpha$.

By~\cite[Prop.~5.3]{AK}, there is an algebra morphism
$\chi_0 : \BB_H^{\alpha} \to k$ such that
\[
\chi_0 \bigl( \sigma(x,y) \bigr) = \alpha(x,y)
\quad\text{and}\quad
\chi_0 \bigl( \sigma^{-1}(x,y) \bigr) = \alpha^{-1}(x,y)
\]
for all $x,y \in H$. 
Consider the maximal ideal $\mm_0 = \Ker (\chi_0 : \BB_H^{\alpha} \to k)$ 
of~$\BB_H^{\alpha}$. 
According to~\cite[Prop.~6.2]{AK},
there is an isomorphism of $H$-comodule algebras
\begin{equation*}
\AA_H^{\alpha}/\mm_0 \, \AA_H^{\alpha} \cong {}^{\alpha} H \, .
\end{equation*}
Thus, ${}^{\alpha} H$ is a central specialization of~$\AA_H^{\alpha}$, 
and $\AA_H^{\alpha}$ is a flat deformation
of~${}^{\alpha} H$ over the commutative algebra~$\BB_H^{\alpha}$.

What about the other central specializations
of~$\AA_H^{\alpha}$, that is the quotients $\AA_H^{\alpha}/\mm \, \AA_H^{\alpha}$,
where $\mm$ is an arbitrary maximal ideal of~$\BB_H^{\alpha}$?
To answer this question, we need the following terminology.

Let $\beta : H \times H \to K$ be an invertible right two-cocycle
with values in a field~$K$ containing the ground field~$k$.
We say that the twisted comodule algebra~$K \otimes_k {}^{\beta} H$
is a \emph{$K$-form} of ${}^{\alpha} H$ if
there is a field~$L$ containing~$K$ and 
an $L$-linear isomorphism of comodule algebras
\[
L \otimes_K (K \otimes_k {}^{\beta} H) \cong L\otimes_k {}^{\alpha} H \, .
\]

By~\cite[Sect.~7]{AK},
for any $K$-form $K \otimes_k {}^{\beta} H$ of ${}^{\alpha} H$,
there exist an algebra morphism $\chi : \BB_H^{\alpha} \to K$ 
and a $K$-linear isomorphism of~$H$-comodule algebras
\[
K_{\chi} \otimes_{\BB_H^{\alpha}} \AA_H^{\alpha} \cong K \otimes_k {}^{\beta} H \, .
\]
Here $K_{\chi}$ stands for~$K$ equipped with the  
$\BB_H^{\alpha}$-module structure induced by
the algebra morphism $\chi : \BB_H^{\alpha} \to K$.
We have 
\[
K_{\chi} \otimes_{\BB_H^{\alpha}} \AA_H^{\alpha} 
\cong \AA_H^{\alpha}/\mm_{\chi} \, \AA_H^{\alpha} \, ,
\]
where $\mm_{\chi} = \Ker(\chi : \BB_H^{\alpha} \to K)$.

The converse is the following statement:
for any field~$K$ containing~$k$ and any algebra morphism
$\chi : \BB_H^{\alpha} \to K$, the $H$-comodule $K$-algebra
\[
K_{\chi} \otimes_{\BB_H^{\alpha}} \AA_H^{\alpha} 
= \AA_H^{\alpha}/ \mm_{\chi} \, \AA_H^{\alpha}
\]
is a $K$-form of~${}^{\alpha} H$.

Theorem~7.2 of~\cite{AK} asserts that this converse holds if 
the algebra $S(t_H)_{\Theta}$ is integral over the subalgebra~$\BB_H^{\alpha}$.
Using~\cite[Th.~3 in Sect.~(4.D)]{Mat} or~\cite[Th.~on p.~105]{Wa}, 
we see that the proof of~\cite[Th.~7.2]{AK} 
works in the more general case where $S(t_H)_{\Theta}$ 
is a \emph{faithfully flat} $\BB_H^{\alpha}$-module
(in particular, if it is a free module).

\section{Reduction to the trivial cocycle}\label{reduction}

The aim of this section is to reduce the previous situation to the case
of the trivial cocycle.

Let us start again with a Hopf algebra~$H$ and an invertible two-cocycle
$\alpha : H \times H \to k$.
Let $L = {}^{\alpha} H ^{\alpha^{-1}}$ be the cocycle deformation of~$H$ by~$\alpha$:
as a coalgebra, $L = H$, and the product~$*$ of~$L$ is given for $x,y \in H$ by
\begin{equation}\label{prodL}
x * y = \alpha(x_1,y_1) \, x_2 y_2 \, \alpha^{-1}(x_3,y_3) \, .
\end{equation}

Let $\eps : L \times L \to k$ be the \emph{trivial cocycle} on~$L$
given for all $x,y \in L$ by 
$$\eps(x,y) = \eps(x) \eps(y)\, .$$

On the level of generic base algebras we have the following result.

\begin{prop}\label{prop-reduction1}
We have $\BB_H^{\alpha} = \BB_L^{\eps}$ inside~$S(t_H)_{\Theta}$.
\end{prop}

\begin{proof}
Let us denote the generic cocycle attached to~$\alpha$ by $\sigma_{\alpha}$
and the generic cocycle attached to~$\eps$ by $\sigma_{\eps}$.
Observe that for $x,y \in H$, we have
\begin{equation*}
\sigma_{\eps}(x,y) 
= t_{x_1} \, t_{y_1} \, \eps(x_2) \, \eps(y_2) \, t^{-1}_{x_3 * y_3}
=  t_{x_1} \, t_{y_1} \, t^{-1}_{x_2 * y_2} \, .
\end{equation*}
We then have
\begin{eqnarray*}
\sigma_{\alpha}(x,y) 
& = & t_{x_1} \, t_{y_1} \, \alpha(x_2, y_2) \, t^{-1}_{x_3 y_3} 
= t_{x_1} \, t_{y_1} \, t^{-1}_{\alpha(x_2, y_2) \, x_3 y_3} \\
& = & t_{x_1} \, t_{y_1} \, t^{-1}_{\alpha(x_2, y_2) \, x_3 y_3 \, \alpha^{-1}(x_4, y_4)} \, \alpha(x_5, y_5) \\
& = & t_{x_1} \, t_{y_1} \, t^{-1}_{x_2 * y_2} \, \alpha(x_3, y_3) \\
& = & \sigma_{\eps}(x_1,y_1)  \, \alpha(x_2, y_2) \, .
\end{eqnarray*}
Similarly,
$\sigma^{-1}_{\alpha}(x,y) = \alpha^{-1}(x_1, y_1) \, \sigma^{-1}_{\eps}(x_2,y_2)$.
This proves that the subalgebras $\BB_H^{\alpha}$ and~$\BB_L^{\eps}$ coincide.
\end{proof}

\begin{corollary}\label{coro-LH}
For any cocommutative Hopf algebra~$H$, we have $\BB_H^{\alpha} = \BB_H^{\eps}$.
\end{corollary}

\begin{proof}
Under the hypothesis, the product~\eqref{prodL} of~$L$ reduces to the product on~$H$:
$x*y = xy$ for all $x,y \in H$.
\end{proof}

Recall that there is a one-to-one correspondence $A \mapsto A^{\alpha}$ 
between the right $L$-comodule algebras
and the right $H$-comodule algebras; it is defined as follows: 
if $A$ is a right $L$-comodule algebras with coaction
$
a \mapsto a_0 \otimes a_1
$
($a \in A$), then $A^{\alpha}$ is $A$ as an $H$-comodule (which is the same as an $L$-comodule)
with the twisted product
\begin{equation*}
a *_{\alpha} b = a_0 \, b_0 \, \alpha(a_1,b_1) \, .
\end{equation*}

The following result for generic Galois extensions is a 
companion to Proposition~\ref{prop-reduction1}.

\begin{prop}\label{prop-reduction2}
(a) We have $\AA_H^{\alpha} = (\AA_L^{\eps})^{\alpha}$.

(b) For any field extension~$K/k$ there is a one-to-one correspondence 
between the $K$-forms of the $H$-comodule algebra~${}^{\alpha} H$
and the $K$-forms of the $L$-comodule algebra~$L$.
\end{prop}

\begin{proof}
(a) By definition, $\AA_H^{\alpha} = \BB_H^{\alpha} \otimes {}^{\sigma_{\alpha}} H$
and $\AA_L^{\eps} = \BB_L^{\eps} \otimes {}^{\sigma_{\eps}} H$.
Let $b,c \in \BB_L^{\eps} = \BB_H^{\alpha}$ and $x,y \in H=L$.
Let us compute the product $(b\otimes x) (c\otimes y)$ in~$(\AA_L^{\eps})^{\alpha}$.
We have
\begin{eqnarray*}
(b\otimes x) *_{\alpha} (c\otimes y) 
& = & bc \, \sigma_{\eps}(x_1, y_1) \otimes x_2 * y_2 \, \alpha(x_3,y_3) \\
& = & bc \, \sigma_{\eps}(x_1, y_1) \otimes \alpha(x_2,y_2) \, x_3 y_3 \, 
\alpha^{-1}(x_4,y_4)\, \alpha(x_5,y_5) \\
& = & bc \, \sigma_{\eps}(x_1, y_1) \, \alpha(x_2,y_2) \otimes x_3 y_3 \\
& = & bc \, \sigma_{\alpha}(x_1, y_1) \otimes x_2 y_2 \, ,
\end{eqnarray*}
which is the product in~$\AA_H^{\alpha}$.
For the last equality, we have used the formula expressing~$\sigma_{\alpha}$
in terms of~$\sigma_{\eps}$ established in the proof of Proposition~\ref{prop-reduction1}.

(b) The above equivalence is compatible with base extension.
\end{proof}

\begin{rems}
(a) Corollary~\ref{coro-LH} holds under the following weaker condition.
Recall that a two-cocycle~$\alpha$ on~$H$ is \emph{lazy} 
in the sense of~\cite{BC} if for all $x,y \in H$,
\begin{equation*}
\alpha(x_1, y_1) \, x_2 y_2 =  \alpha(x_2, y_2) \, x_1 y_1\, .
\end{equation*}
If $\alpha$ is lazy, then
obviously the Hopf algebra ${}^{\alpha} H ^{\alpha^{-1}}$ coincides with~$H$.
This includes the case where $H$ is cocommutative
since all two-cocycles on a cocommutative Hopf algebra are obviously lazy.

(b) In the case of a general Hopf algebra~$H$,
one can check that if $\alpha$ and~$\beta$ are cohomologous two-cocycles,
then the twisted Hopf algebras ${}^{\alpha} H ^{\alpha^{-1}}$
and~${}^{\beta} H ^{\beta^{-1}}$ are isomorphic, the isomorphism
$f: {}^{\alpha} H ^{\alpha^{-1}} \to {}^{\beta} H ^{\beta^{-1}}$ being given by
\begin{equation*}
f(x) = \lambda^{-1}(x_1) \, x_2 \, \lambda(x_3)\, ,
\end{equation*}
where $\lambda$, $\lambda^{-1}$ are the linear forms appearing in~\eqref{cohomologous};
see also~\cite[Prop.~5.4]{AK}.

(c) It follows from the previous discussion that if $H$ is such that
any two-cocycle on~$H$ is cohomologous to a lazy one, then
${}^{\alpha} H ^{\alpha^{-1}} \cong H$ for any~$\alpha$.
This holds for instance for Sweedler's four-dimensional Hopf algebra.

(d) Examples of cocycles $\alpha$ for which the Hopf algebras $H$ and $^{\alpha}H^{\alpha^{-1}}$ 
are not isomorphic can be drawn from the proofs of~\cite[Th.~4.5]{KS} or~\cite[Th.~4.3]{Ma2}.
These examples include the quantized enveloping algebras and the naturally associated graded pointed Hopf algebras.
\end{rems}

\section{Flatness and freeness results}\label{results}

From now on we consider a Hopf algebra~$H$ together with the trivial cocycle~$\eps$.
As we have seen in the proof of Proposition~\ref{prop-reduction1}, 
the generic cocycle~$\sigma$ attached to~$\eps$ and its inverse~$\sigma^{-1}$ are given 
for all $x,y \in H$ by
\begin{equation}\label{triv-generic}
\sigma(x,y) 
=  t_{x_1} \, t_{y_1} \, t^{-1}_{x_2 y_2} 
\quad\text{and}\quad
\sigma^{-1}(x,y) = t_{x_1 y_1} \, t^{-1}_{x_2}  \, t^{-1}_{y_2} \, .
\end{equation}

To simplify notation, we write~$\BB_H$ for~$\BB_H^{\eps}$,
and~$\AA_H$ for~$\AA_H^{\eps}$.

\subsection{On the subalgebra~$\BB_H$}\label{structure}

We give some general properties of~$\BB_H$ as a subalgebra of~$S(t_H)_{\Theta}$,
including flatness.

Let us start by analyzing the difference between the two algebras. To this end
we consider the ideal~$(\BB_H^+)$ of~$S(t_H)_{\Theta}$
generated by the augmentation ideal~$\BB_H^+ = \BB_H \cap \Ker(\eps)$. Let 
\begin{equation*}
S(t_H)_{\Theta}/(\BB_H^+)
\end{equation*}
be the quotient algebra. 

In order to identify this quotient, 
we consider the largest commutative quotient algebra~$H_{\ab}$ of~$H$:
this is the quotient of~$H$ by the two-sided ideal generated by all commutators
$xy - yx$ ($x,y \in H$). It is easily checked that $H_{\ab}$ is a quotient Hopf algebra of~$H$;
another way to see this is to observe that $H_{\ab}$ represents the group-valued functor
\[
R \longmapsto \Alg_k(H,R)
\]
on the category of commutative algebras~$R$; see~\cite{Wa}.

We can now state the following.

\begin{prop}\label{quotient}
The canonical coalgebra morphism $t_H \to H \, ; \; t_x \mapsto x$, 
composed with the projection $H \to H_{\ab}$, 
induces a Hopf algebra epimorphism
$S(t_H)_{\Theta} \to H_{\ab}$.  
This epimorphism has~$(B_H)^+$ as kernel.
Therefore, $(B_H)^+$  is a Hopf ideal of ~$S(t_H)_{\Theta}$, and we have
an isomorphism 
\[
S(t_H)_{\Theta}/(\BB_H^+) \cong H_{\ab} 
\]
of commutative Hopf algebras.
\end{prop}

\begin{proof}
For $x\in H$ let $t'_x$ be the image of~$t_x$ in~$Q = S(t_H)_{\Theta}/(\BB_H^+)$.
Since $t_{x_1} \, t_{y_1} \, t^{-1}_{x_2 y_2} - \eps(x)\eps(y)$ belongs to~$\BB_H^+$,
we have
\[
t'_{xy} = t'_x \, t'_y \in Q
\]
for all $x, y\in H$.
The algebra morphism $H \to Q \, ; \, x\mapsto t'_x$ takes values in a commutative algebra; 
therefore, it factors through~$H_{\ab}$. 
It thus induces an algebra morphism $H_{\ab} \to Q$. 

On the other hand, by the universal property of~$S(t_H)_{\Theta}$, 
the canonical Hopf algebra morphism $H \to H_{\ab}\, ; \, x\mapsto \bar x$ 
gives rise to a Hopf algebra morphism $S(t_H)_{\Theta} \to H_{\ab}$
induced by $t_x \mapsto \bar x$. 
Since $\overline{x_1} \,\overline{y_1} \, S(\overline{x_2 y_2}) =  \eps(x)\eps(y)$ in~$H_{\ab}$,
the morphism $t_x \mapsto \bar x$ vanishes on~$\BB_H^+$, yielding a Hopf algebra morphism
$Q \to H_{\ab}$. 
It is easy to chek that the latter is the inverse of the morphism $H_{\ab} \to Q$ constructed above.
\end{proof}

We next identify a generating set for~$S(t_H)_{\Theta}$ as a $\BB_H$-module.

\begin{lemma}\label{generators}
As a $\BB_H$-module, $S(t_H)_{\Theta}$ is generated by
the elements~$t_x$, where $x$~runs over~$H$.
\end{lemma}

\begin{proof}
Let $M$ be the $\BB_H$-submodule of~$S(t_H)_{\Theta}$
generated by the elements~$t_x$ ($x\in H$).
We wish to prove that $M = S(t_H)_{\Theta}$.
Since $S(t_H)_{\Theta}$ is generated by the elements~$t^{\pm 1}_x$
as an algebra, it suffices to check that each product $t_x \, t_y$ and each inverse~$t^{-1}_x$
are linear combinations of elements~$t_z$ ($z\in H$) with coefficients in~$\BB_H$.
For the product $t_x \, t_y$, we have
\begin{equation*}\label{txty}
\sigma(x_1,y_1) \, t_{x_2y_2}
= t_{x_1} \, t_{y_1} \, t^{-1}_{x_2y_2} \,  t_{x_3y_3} 
= t_{x_1} \, t_{y_1} \, \eps (x_2) \, \eps(y_2) 
= t_x \, t_y  \, .
\end{equation*}
This shows that $t_x t_y$ belongs to~$M$.

Let us now check that $t^{-1}_x$ belongs to~$M$ for all $x\in H$.
By~\cite[Lemma~5.1]{AK} we have $t^{-1}_1 = \sigma^{-1}(1,1) \in \BB_H$.
Next, we claim that $t^{-1}_{S(x_1)} \, t^{-1}_{x_2}$ belongs to~$\BB_H$.
Indeed, the element $\sigma^{-1}(S(x_1),x_2) \, t^{-1}_1$ of~$\BB_H$ 
can be rewritten as
\begin{eqnarray*}
\sigma^{-1}(S(x_1),x_2) \, t^{-1}_1
& = & t_{S(x_2) x_3} \, t^{-1}_{S(x_1)}  \, t^{-1}_{x_4} \, t^{-1}_1 \\
& = & \eps(x_2) \, t_1 \, t^{-1}_{S(x_1)}  \, t^{-1}_{x_3} \, t^{-1}_1 \\
& = & t^{-1}_{S(x_1)}  \, t^{-1}_{x_2} \, ,
\end{eqnarray*}
which proves the claim. Using the latter, we obtain
\begin{equation}\label{tinv}
t^{-1}_x = \eps(x_1) \, t^{-1}_{x_2} = \eps(S(x_1)) \, t^{-1}_{x_2} 
= t_{S(x_1)} \, (t^{-1}_{S(x_2)} \, t^{-1}_{x_3})
\in \sum_{y\in H} \, \BB_H \, t_y \, .
\end{equation}
This shows that~$t^{-1}_x$ belongs to~$M$.  
\end{proof}

Let us compute the coproduct and the counit on the values of~$\sigma$.

\begin{lemma}\label{coprod-generic}
In the commutative Hopf algebra~$S(t_H)_{\Theta}$ we have
\begin{equation*}
\Delta \bigl( \sigma(x,y) \bigr) = t_{x_1} \, t_{y_1} \, t^{-1}_{x_3 y_3} \otimes  \sigma(x_2,y_2) \, ,
\end{equation*}
\begin{equation*}
\Delta \bigl( \sigma^{-1}(x,y) \bigr) 
=  t_{x_1y_1} \, t^{-1}_{x_3} \, t^{-1}_{y_3} \otimes  \sigma^{-1}(x_2,y_2) \, ,
\end{equation*}
and 
\begin{equation*}
\eps( \sigma(x,y) )  =  \eps( \sigma^{-1}(x,y) ) = \eps(x) \, \eps(y)
\end{equation*}
for all $x,y \in H$.
\end{lemma}

\begin{proof}
Let us check the first formula. 
By~\eqref{comm-coalg} and~\eqref{comm-Hopf}, 
\begin{eqnarray*}
\Delta \bigl( \sigma(x,y) \bigr) 
& = & \Delta(t_{x_1}) \, \Delta(t_{y_1}) \, \Delta(t^{-1}_{x_3 y_3}) \\
& = & (t_{x_1} \otimes t_{x_2}) \, (t_{y_1} \otimes t_{y_2}) \, 
(t^{-1}_{x_4 y_4} \otimes t^{-1}_{x_3 y_3}) \\
& = & t_{x_1} \, t_{y_1} \, t^{-1}_{x_4 y_4}\otimes t_{x_2} \, t_{y_2} \, t^{-1}_{x_3 y_3} \\
& = & t_{x_1} \, t_{y_1} \, t^{-1}_{x_3 y_3}\otimes \sigma(x_2,y_2) \, .
\end{eqnarray*}
A similar computation yields the remaining formulas.
\end{proof}

As a consequence, we obtain the following.

\begin{prop}\label{coideal}
For any Hopf algebra~$H$, 
the commutative algebra~$\BB_H$ is a left coideal subalgebra of~$S(t_H)_{\Theta}$.

If in addition $H$ is cocommutative, then $\BB_H$ is a Hopf subalgebra of~$S(t_H)_{\Theta}$.
\end{prop}

\begin{proof}
Since $\BB_H$ is generated by the values of~$\sigma$ and~$\sigma^{-1}$, the first assertion
is a consequence of Lemma~\ref{coprod-generic}.

When $H$ is cocommutative, then so is~$S(t_H)_{\Theta}$,
and $\Delta \bigl( \sigma(x,y) \bigr)$ can be rewritten as
\begin{equation*}
\Delta \bigl( \sigma(x,y) \bigr) 
=  t_{x_1} \, t_{y_1} \, t^{-1}_{x_2 y_2} \otimes  \sigma(x_3,y_3)
=  \, \sigma(x_1,y_1) \otimes  \sigma(x_2,y_2) .
\end{equation*}
Similarly, we have
\begin{equation*}
\Delta \bigl( \sigma^{-1}(x,y) \bigr) 
=  \sigma^{-1}(x_1,y_1)  \otimes  \sigma^{-1}(x_2,y_2)  \, .
\end{equation*}
This implies that $\BB_H$ is a sub-bialgebra of~$S(t_H)_{\Theta}$.

It remains to show that $\BB_H$ is stable under the antipode~$S$ of~$S(t_H)_{\Theta}$.
By~\eqref{ant-comm-coalg} and~\eqref{triv-generic} we have
\begin{eqnarray*}
S \bigl( \sigma(x,y) \bigr) 
& = & S(t^{-1}_{x_2y_2}) \, S(t_{x_1}) \, S(t_{y_1})
= t_{x_2y_2} \, t^{-1}_{x_1} \, t^{-1}_{y_1} \\
& = & t_{x_1y_1} \, t^{-1}_{x_2} \, t^{-1}_{y_2}
= \sigma^{-1}(x,y) \, .
\end{eqnarray*}
Since $S(t_H)_{\Theta}$ is cocommutative, the antipode is involutive and,
as a consequence of the previous computation, we obtain
\begin{equation*}
S \bigl( \sigma^{-1}(x,y) \bigr) = \sigma(x,y) \, .
\end{equation*}
This proves that $\BB_H$ is a Hopf subalgebra of~$S(t_H)_{\Theta}$.
\end{proof}

In general, a commutative Hopf algebra $A$ is flat over any left (or right)
coideal subalgebra $B$ (see \cite[Th.~3.4]{MW}).
In view of Proposition~\ref{coideal}, we can apply this to $A = S(t_H)_{\Theta}$
and $B = \BB_H$, to obtain the following.

\begin{theorem}\label{flat}
For any Hopf algebra~$H$,
the $\BB_H$-module~$S(t_H)_{\Theta}$ is flat.
\end{theorem}

\subsection{Faithful flatness}

We give sufficient conditions for $S(t_H)_{\Theta}$ to be faithfully flat over $\BB_H$.
This faithful flatness is equivalent to the condition that $S(t_H)_\Theta$
is a projective generator as a $\BB_H$-module, since in general,
a commutative Hopf algebra is faithfully flat over a left (or right)
coideal subalgebra if and only if it is a projective generator
(see \cite[Cor.~3.5]{MW} and also \cite[Th.~2.1]{MW}).

Firstly, we show that the algebras have stronger properties 
in the finite-dimensional case.

\begin{theorem}\label{thm-finite}
Let $H$ be a finite-dimensional Hopf algebra. Then
$\BB_H$ is Noetherian and
$S(t_H)_{\Theta}$ is a finitely generated projective $\BB_H$-module.
Moreover, it is a generator as a $\BB_H$-module.
\end{theorem}

It follows that $S(t_H)_{\Theta}$ is integral over~$\BB_H$. 
This answers Question~6.1 in~\cite{Ka1} when~$H$ is finite-dimensional.
Consequently, the results of~\cite[Sect.~7]{AK} hold
unconditionally for any finite-dimensional Hopf algebra.

\begin{proof}
Since the algebra~$\BB_H$ is generated 
by the finitely many elements $\sigma^{\pm 1}(x_i,x_j)$,
where $(x_i)_i$ is an arbitrarily chosen linear basis of~$H$, it is Noetherian.

Next, it follows from Lemma~\ref{generators} that 
$S(t_H)_{\Theta}$ is generated as a $\BB_H$-module
by the elements~$t_{x_i}$, where $x_i$ runs through the chosen finite basis. 
Therefore, $S(t_H)_{\Theta}$ is finitely generated as a $\BB_H$-module.
Since $\BB_H$ is Noetherian, 
it follows that the $\BB_H$-module $S(t_H)_{\Theta}$ is finitely presented.

We know from Theorem~\ref{flat} that $S(t_H)_{\Theta}$ is a flat $\BB_H$-module. 
This together with the finite presentation assertion implies that the module is finitely generated projective.
The last assertion in the theorem is a consequence of the previous one.
\end{proof}

\begin{corollary}
If $H$ is a finite-dimensional Hopf algebra, 
then the Krull dimension of~$\BB_H$ is equal to~$\dim_k H$.
\end{corollary}

In other words, the affine algebraic variety $\Spec (\BB_H)$ has dimension~$\dim_k H$.

\begin{proof}
The Krull dimension of~$S(t_H)_{\Theta}$ is clearly equal to~$\dim_k H$.
Since $S(t_H)_{\Theta}$ is integral over~$\BB_H$, 
it follows from~\cite[Th.~20 in Sect.~(13.C)]{Mat} that $\BB_H$ and $S(t_H)_{\Theta}$ 
have the same Krull dimension. 
\end{proof}

Secondly, we have the following.

\begin{theorem}\label{thm-cocom}
If $H$ is a cocommutative Hopf algebra, 
then $S(t_H)_\Theta$ is faithfully flat over $\BB_H$.
\end{theorem}

\begin{proof}
Takeuchi proved that a cocommutative Hopf algebra~$A$ is a projective generator
over any Hopf subalgebra~$B$ (see~\cite[Remark~3.6]{MW}).
In view of Proposition~\ref{coideal}, we may apply this 
to $A = S(t_H)_{\Theta}$ and~$B = \BB_H$, to obtain the desired result.
\end{proof}

\subsection{Freeness}

We now give sufficient conditions for $S(t_H)_{\Theta}$ to be free over~$\BB_H$.
Let us start with the following.

\begin{theorem}\label{thm-pointed}
If $H$ is a pointed Hopf algebra such that 
each element of the kernel of the natural epimorphism
\[
G(H)_{\ab} \to G(H_{\ab})
\]
is of finite order,
then the algebra $S(t_H)_{\Theta}$ is a free $\BB_H$-module.
\end{theorem}

Here $G(H)_{\ab}$ denotes the abelianization of the group~$G(H)$.
The finiteness condition in the theorem
is satisfied for instance if all elements of~$G(H)_{\ab}$
are of finite order, or if $G(H)_{\ab}$ is trivial.
Theorem~\ref{thm-pointed} generalizes~\cite[Th.~6.2]{Ka1}.

In order to prove the theorem, we need the following preliminaries. 
If $H$ is pointed, then so is $S(t_H)_{\Theta}$, and its grouplike elements
form the free abelian group on the grouplike elements of~$H$:
\begin{equation*}
G(S(t_H)_{\Theta}) = \langle t_g \, |\, g\in G(H) \rangle_{\ab} \, .
\end{equation*}
Let $Q = S(t_H)_{\Theta}/(\BB_H^+)$ 
be the quotient Hopf algebra considered in Proposition~\ref{quotient}.
We denote by
\begin{equation*}
C =  S(t_H)_{\Theta}^{\co-Q}
\end{equation*}
the left coideal subalgebra of~$S(t_H)_{\Theta}$ consisting of all $Q$-coinvariants. 
Then $C$ contains~$\BB_H$.
Let
\begin{equation*}
G(C) =  C \cap G(S(t_H)_{\Theta}) \quad\text{and}\quad
G(\BB_H) =  \BB_H \cap G(S(t_H)_{\Theta})
\end{equation*}
be the submonoids of~$G(S(t_H)_{\Theta})$ consisting of the grouplike elements
contained in~$C$ and in~$\BB_H$, respectively. 
By definition of~$C$,
\begin{equation}\label{GC}
G(C) =  \Ker \bigl( G(S(t_H)_{\Theta}) \to G(Q) \bigr) \, ,
\end{equation}
which is the kernel of the group homomorphism (indeed, epimorphism) induced by the projection 
$S(t_H)_{\Theta} \to Q \cong H_{\ab}$. In particular, $G(C)$ is a group.

\begin{lemma}\label{lem1}
The group $G(C)$ is the subgroup of $G(S(t_H)_{\Theta})$ generated by all elements of the form
$t_gt_h t^{-1}_{gh}$, where $g,h \in G(H)$, and of the form~$t_f$,
where $f$ belongs to the kernel of $G(H) \to G(H_{\ab})$.
\end{lemma}

\begin{proof}
Let $G_0$ be the subgroup of~$G(S(t_H)_{\Theta})$ generated by the elements
$t_gt_h t^{-1}_{gh}$, where $g,h \in G(H)$.
Since 
\begin{equation}\label{modG0}
t_{gh} \equiv t_g t_h \, \quad t_1 \equiv 1\, , \quad
t_{g^{-1}} \equiv t^{-1}_g \, , \quad 
t_{ghg^{-1}h^{-1}} \equiv 1
\end{equation}
modulo~$G_0$, we have the isomorphism
\begin{equation*}
G(S(t_H)_{\Theta})/G_0 \cong G(H)_{\ab} \, .
\end{equation*}

Now let $G_1$ be the smallest subgroup containing~$G_0$ 
and the elements~$t_f$ described in the lemma.
We have the isomorphism
\begin{equation*}
G(S(t_H)_{\Theta})/G_1  \cong G(H_{\ab}) \, .
\end{equation*}
The lemma follows from this isomorphism and from~~\eqref{GC}.
\end{proof}

The following is a consequence of~\cite[Th.~1.3]{Ma1}.

\begin{lemma}\label{lem2}
If $H$ is a pointed Hopf algebra, then the following are equivalent:

(a) $\BB_H \hookrightarrow S(t_H)_{\Theta}$ is faithfully flat;

(b) $S(t_H)_{\Theta}$ is a free $\BB_H$-module;

(c) $G(\BB_H)$ is closed under inverse;

(d) $G(\BB_H) = G(C)$;

(e) $\BB_H = C$.
\end{lemma}

\begin{proof}[Proof of Theorem~\ref{thm-pointed}]
We will verify Condition\,(c) of Lemma~\ref{lem2}. 
Let $G_0$ be as in the proof of Lemma~\ref{lem1}. 
Pick any element $\gamma \in G(\BB_H)$; we will prove that it has an inverse in~$G(\BB_H)$.
Since $G(\BB_H) \subset G(C)$, it follows from Lemma~\ref{lem1} and from~\eqref{modG0}
that 
\[
\gamma = \beta \, t_f
\]
for some $\beta \in G_0$ and $f\in \Ker(G(H) \to G(H_{\ab}))$.
We must prove $t_f^{-1} \in \BB_H$.
If the image of $f$ in the kernel of $G(H)_{\ab} \to G(H_{\ab})$
has order $N < \infty$, then
$f^N \equiv 1$ modulo $[G(H),G(H)]$.
By~\eqref{modG0}, this implies that $t_f^N \equiv 1$, and so
$t_f^{-1} \equiv t_f^{N-1}$ modulo~$G_0$.
Since $t_f^{N-1}$ belongs to~$\BB_H$, so does~$t_f^{-1}$, as desired.
\end{proof}

\begin{Qu}
Suppose that $H$ is pointed. If $t_g\in \BB_H$ for some $g\in G(H)$, 
does $t^{-1}_g$ belong to~$\BB_H$, or equivalently, 
does $t_{g^{-1}}$ belong to~$\BB_H$?
If this is true, then following the proof of Theorem~\ref{thm-pointed},
we can conclude that $S(t_H)_{\Theta}$ is a free $\BB_H$-module
for any pointed Hopf algebra~$H$.
\end{Qu}

We now give more sufficient conditions for $S(t_H)_{\Theta}$ to be free over~$\BB_H$.

\begin{theorem}\label{thm-com}
Let $H$ be a Hopf algebra.
If ~$H$ is commutative or pointed cocommutative, then
the algebra $S(t_H)_{\Theta}$ is a free $\BB_H$-module.
\end{theorem}

Thus the conclusions of Theorem~7.2 and Corollary~7.3 of~\cite{AK} hold true
whenever~$H$ satisfies one of the hypotheses of 
Theorems~\ref{thm-finite},~\ref{thm-cocom},~\ref{thm-pointed}, and~\ref{thm-com}.

\begin{proof}
(a) Assume that $H$ is commutative.
We shall prove that $\AA_H= S(t_H)_{\Theta}$;
since $\AA_H = \BB_H \otimes H$ as a $\BB_H$-module, this implies the desired freeness.
By the commutativity of~$H$, the identity map of~$H$ induces a Hopf algebra
epimorphism $S(t_H)_{\Theta} \to H$, by which we regard $S(t_H)_{\Theta}$ as a right $H$-comodule algebra.
Since the natural map $H \to S(t_H)_{\Theta}$ is a coalgebra section of the epimorphism above,
it follows from the cleftness theorem of Doi and Takeuchi~\cite[Th.~9]{DT} that the section
induces an isomorphism
\begin{equation*}
S(t_H)_{\Theta}^{\co-H} \otimes {}^{\sigma} H \overset{\cong}{\longrightarrow} S(t_H)_{\Theta}
\end{equation*}
of $H$-comodule algebras over~$S(t_H)_{\Theta}^{\co-H}$.
Here, the left-hand side denotes the twisted product associated to the generic two-cocycle~$\sigma$.
Therefore, $S(t_H)_{\Theta}^{\co-H}$ contains~$\BB_H$, and $S(t_H)_{\Theta}$
naturally contains $\AA_H = \BB_H \otimes {}^{\sigma} H$ as an $H$-comodule subalgebra.
But $\AA_H$ contains the generators of~$S(t_H)_{\Theta}$
(note from~\eqref{tinv} that $t_x^{-1}$ is contained in~$\AA_H$).
Therefore, $\AA_H= S(t_H)_{\Theta}$, as desired.

(b) Assume that $H$ is pointed cocommutative. 
Then by Kostant's theorem (see~\cite[Sect.~8.1]{Sw})
the inclusion $kG(H) \hookrightarrow H$ splits as a Hopf algebra morphism.
It follows that 
\[
k[G(H)_{\ab}] = ( kG(H) )_{\ab} \to H_{\ab}
\]
is a split monomorphism;
hence the natural epimorphism $G(H)_{\ab} \to G(H_{\ab})$ is an isomorphism. 
We can then apply Theorem~\ref{thm-pointed}.
\end{proof}

\begin{rems}
(a) Theorem~\ref{thm-com} also holds true if
$H$ is a pointed Hopf algebra such that
the inclusion $kG(H) \hookrightarrow H$ splits as an algebra morphism.
In this case the proof above works as well.

(b) If $\BB_H \hookrightarrow S(t_H)_{\Theta}$ is \emph{faithfully flat}, 
then by~\cite[Th.~3]{Ta2} we have
\begin{equation}\label{BSQ}
\BB_H = S(t_H)_{\Theta}^{\co-Q} \, ,
\end{equation}
where $Q = S(t_H)_{\Theta}/(\BB_H^+)$ as above.
It follows that if $H_{\ab} \cong Q$ is the trivial one-dimensional Hopf algebra,
then $\BB_H = S(t_H)_{\Theta}$. This applies for instance to the case where
$G$ is a group with \emph{trivial abelianization}
(e.g., to a non-abelian simple group), so that in this case,
\[
\BB_{kG} = S(t_{kG})_{\Theta} = k[\, t_g^{\pm 1} \, |\, g\in G \, ] \, ,
\]
the Laurent polynomial algebra in the indeterminates~$t_g$ ($g\in G$).

(c)  Let $G$ be a \emph{finite} group and $H = O_k(G)$ the dual Hopf algebra of the group algebra~$kG$.
The elements of~$O_k(G)$ can be seen as $k$-valued fonctions on~$G$.
In this case, 
\[
S(t_H)_{\Theta} = k[\, t_g \, |\, g\in G \, ] \Bigl[\frac{1}{\Theta_G} \Bigr] \, ,
\]
where $\Theta_G = \det(t_{gh^{-1}})_{g,h \in G}$ is \emph{Dedekind's group determinant}
(see \cite[Ex.~B.5]{AK}).
Since $Q \cong H_{\ab} = H$ and since $H$-coinvariants identify with $G$-invariants,
it follows from~\eqref{BSQ} above that $\BB_H$ is the subalgebra of
$G$-invariant elements of~$S(t_H)_{\Theta}$:
\[
\BB_{kG} = \bigl(S(t_H)_{\Theta}\bigr)^G 
= k[\, t_g \, |\, g\in G \, ]{\,}^G \Bigl[\frac{1}{\Theta_G^2} \Bigr] \, ,
\]
where $G$ acts on the elements~$t_g$ by translation.
Incidentally, the lazy two-cocycles have been classified in this case by~\cite{GK}.
\end{rems}

\section{Another set of generators for $\BB_H$}\label{gens}

In~\cite[Sect.~7]{Ka1} a systematic method to exhibit elements of~$\BB_H$ was given.
Let us recall it.

Let $H$ be a  Hopf algebra and
$X_H$ a copy of the underlying vector space of~$H$;
we denote the identity map from $H$ to~$X_H$ by $x\mapsto X_x$ 
for all $x\in H$.
Consider the tensor algebra~$T(X_H)$.
If $\{x_i\}_{i\in I}$ is a linear basis of~$H$, 
then $T(X_H)$ is the free non-commutative algebra
over the indeterminates $\{X_{x_i}\}_{i\in I}$.
The natural right (resp.\ left) $H$-comodule structure on~$H$ 
extends to a right (resp.\ left) $H$-comodule algebra structure on~$T(X_H)$.

Now, for $x,y \in H$ consider the following elements of~$T(X_H)$:
\begin{equation}\label{P1}
P_x = X_{x_1} \, X_{S(x_2)}  \, ,
\qquad
P'_x = X_{S(x_1)} \, X_{x_2} \, ,
\end{equation}
\begin{equation}\label{P2}
Q_{x,y} = X_{x_1} \, X_{y_1} \, X_{S(x_2 y_2)} \, , \qquad
Q'_{x,y} = X_{S(x_1 y_1)} \, X_{x_2} \, X_{y_2} \, . 
\end{equation}
It is easy to check that $P_x$ and~$Q_{x,y}$ are right coinvariant elements of $T(X_H)$, 
and  that $P'_x$ and~$Q'_{x,y}$ are left coinvariant elements;
see~\cite[Lem\-ma~2.1]{AK} or~\cite[Prop.~7.1]{Ka1}.

The non-commutative algebra $T(X_H)$ can be related to the commutative algebra~$S(t_H)$
of Section~\ref{basic} by the right $H$-comodule algebra morphism
\begin{equation*}
\mu : T(X_H) \longrightarrow  S(t_H) \otimes H \, ; \; X_x \longmapsto  t_{x_1} \otimes x_2 \, .
\end{equation*}
By~\cite[Lemma~4.2]{AK} or~\cite[Prop.~7.2]{Ka1}, 
the map $\mu$ is universal in the sense that any right $H$-comodule algebra morphism
$T(X_H) \to H$ factors through~$\mu$.

Similarly, we can define an anti-algebra morphism 
\[
\mu': T(X_H) \longrightarrow  S(t_H)_{\Theta} \otimes H
\]
by $\mu'(X_x) = t^{-1}_{x_2} \otimes S(x_1)$ for all $x\in H$.
It is easy to check that $\mu'$ is the convolution inverse of~$\mu$.

Consider the following elements of $S(t_H)_{\Theta} \otimes H$:
\begin{equation*}
p_x = \mu(P_x) \, , \quad
p'_x = \mu'(P'_x) \, , \quad
q_{x,y} = \mu(Q_{x,y}) \, , \quad
q'_{x,y} = \mu(Q'_{x,y}) \, .
\end{equation*}

\begin{lemma}\label{lem-pq}
For all $x,y \in H$, we have
\[
p_x = t_{x_1} \, t_{S(x_2)} \otimes 1 \, , \quad
p'_{x} = t^{-1}_{S(x_1)} \, t^{-1}_{x_2} \otimes 1\, , 
\]
\[
q_{x,y} = t_{x_1} \, t_{y_1} \, t_{S(x_2y_2)} \otimes 1\, , \quad
q'_{x,y} = t^{-1}_{S(x_1y_1)} \, t^{-1}_{x_2} \, t^{-1}_{y_2} \otimes 1\, .
\]
\end{lemma}

\begin{proof}
We have
\begin{eqnarray*}
p_x & = & \mu(X_{x_1}) \, \mu(X_{S(x_2)} ) 
=  t_{x_1} \, t_{S(x_4)}  \otimes x_2 S(x_3)\\
& = &  t_{x_1} \, t_{S(x_3)}  \otimes \eps(x_2) \, 1
=  t_{x_1} \, t_{S(x_2)} \otimes 1 \, . 
\end{eqnarray*}
Similar computations work for the other elements.
\end{proof}

It follows that 
the elements $p_x$, $p'_x$, $q_{x,y}$, $q'_{x,y}$ belong to 
the subalgebra $S(t_H)_{\Theta} \otimes 1$ of~$S(t_H)_{\Theta} \otimes H$,
which we identify with~$S(t_H)_{\Theta}$. 
We can therefore drop the suffix~$\otimes 1$ in the previous formulas.

Note that $p_x$ and $q_{x,y}$ belong to the polynomial algebra~$S(t_H)$ generated
by the elements~$t_x$, 
whereas $p'_x$ and $q'_{x,y}$ belong to the subalgebra generated
by the elements~$t^{-1}_x$.

It follows from~\eqref{ant-comm-coalg} and Lemma~\ref{lem-pq}
that if $H$ is \emph{cocommutative}, then $p'_{x}$ and $q'_{x,y}$
are the antipodes of~$p_{x}$ and~$q_{x,y}$, respectively:
\begin{eqnarray*}
S(p'_{x}) = t_{S(x_1)} \, t_{x_2} = p_x
\quad\text{and}\quad 
S(q'_{x,y}) = t_{S(x_1y_1)} \, t_{x_2} \, t_{y_2} = q_{x,y} \, .
\end{eqnarray*}

The following holds for any general Hopf algebra~$H$.

\begin{prop}
For all $x,y \in H$,
the elements $p_x$, $p'_x$, $q_{x,y}$, $q'_{x,y}$ belong to~$\BB_H$.
\end{prop}

This had been observed for the elements $p_x$ and $q_{x,y}$
in~\cite[Prop.~7.3]{Ka1}.

\begin{proof}
It suffices to express these elements in terms of the generators $\sigma^{\pm 1}(x,y)$ of~$\BB_H$.
Using~\eqref{triv-generic}, it is easily checked that
\begin{eqnarray*}
p_x & = & \sigma(x_1, S(x_2)) \, \sigma(1, 1) \, , \\
p'_x & = & \sigma^{-1}(S(x_1),x_2) \, \sigma^{-1}(1,1) \, , \\
q_{x,y} & = & \sigma(x_1, y_1) \, \sigma(x_2y_2, S(x_3y_3)) \, \sigma(1, 1) \, , \\
q'_{x,y} & = & \sigma^{-1}(S(x_1y_1),x_2y_2) \, \sigma^{-1}(x_3,y_3) \, \sigma^{-1}(1,1) \, .
\end{eqnarray*}
(The second equality has already been established in the proof of Lemma~\ref{generators}.)
The conclusion follows.
\end{proof}

We now express the values of~$\sigma$ and of~$\sigma^{-1}$ in terms of the previous elements.

\begin{prop}\label{prop-nice}
For all $x,y \in H$ we have
\[
\sigma(x,y) = q_{x_1,y_1} \, p'_{x_2y_2}
\quad\text{and}\quad
\sigma^{-1}(x,y) = p_{x_1y_1} \, q'_{x_2,y_2} \, .
\]
\end{prop}

\begin{proof}
In view of~\eqref{triv-generic} again, we have
\begin{eqnarray*}
q_{x_1,y_1} \, p'_{x_2y_2}
& = & t_{x_1} \, t_{y_1} \, t_{S(x_2y_2)} \, t^{-1}_{S(x_3y_3)} \, t^{-1}_{x_4y_4} \\
& = & t_{x_1} \, t_{y_1} \, t^{-1}_{x_2y_2} 
= \sigma(x,y)\, .
\end{eqnarray*}
Similarly,
\begin{eqnarray*}
p_{x_1y_1} \, q'_{x_2,y_2}
& = & t_{x_1y_1} \, t_{S(x_2y_2)} \, t^{-1}_{S(x_3y_3)} \, t^{-1}_{x_4} \, t^{-1}_{y_4} \\
& = & t_{x_1y_1} \, t^{-1}_{x_2} \, t^{-1}_{y_2}
= \sigma^{-1}(x,y)\, .
\end{eqnarray*}
\end{proof}

\begin{corollary}
For any Hopf algebra~$H$, the algebra~$\BB_H$ is generated by the elements
$p_x$, $p'_x$, $q_{x,y}$, $q'_{x,y}$ ($x,y\in H$).
\end{corollary}

The theory of polynomial identities worked out in~\cite{AK}
and the specific computations in the case of the Sweedler algebra
(see~\cite[Sect.~10]{AK} and~\cite[Sect.~4 \& Ex.~7.4]{Ka1})
suggest that the generators $p_x$, $p'_x$, $q_{x,y}$, $q'_{x,y}$
are more natural than the values of~$\sigma^{\pm 1}$.

\begin{corollary}\label{coro-coco}
If $H$ is a cocommutative Hopf algebra, then $\BB_H$ is generated by
the elements $p_x$, $q_{x,y}$ and their antipodes.
\end{corollary}

\begin{proof}
In this case we have already observed that $p'_x$ and $q'_{x,y}$ are the antipodes of
$p_x$ and $q_{x,y}$, respectively.
\end{proof}

\begin{rem}
It follows from Corollaries~\ref{coro-LH} and~\ref{coro-coco}
that any non-de\-gen\-er\-ate invertible two-cocycle on a cocommutative Hopf algebra
is \emph{nice} in the sense of~\cite[Sect.~9]{AK}. 
\end{rem}

\section*{Acknowledgements}

The present joint work is part of the project ANR BLAN07-3$_-$183390 
``Groupes quantiques~: techniques galoisiennes et d'int\'egration" funded
by Agence Nationale de la Recherche, France.

\end{document}